\documentclass[letterpaper, 10 pt, conference]{ieeeconf}  

\IEEEoverridecommandlockouts                              

\usepackage{graphics} 
\usepackage{amsmath} 
\usepackage{amssymb}  
\usepackage{xcolor}
\usepackage{amsthm}
\usepackage{graphicx}
\usepackage{arydshln}
\usepackage[font=small,skip=-3pt]{subcaption}
\usepackage[font=footnotesize,skip=2pt]{caption}

\newtheoremstyle{boldspace}
  {\topsep} 
  {\topsep} 
  {} 
  {} 
  {\bfseries} 
  {.} 
  {.5em} 
  {} 

\theoremstyle{boldspace} \newtheorem{remark}{Remark}
\theoremstyle{boldspace} \newtheorem{theorem}{Theorem}
\theoremstyle{boldspace} \newtheorem{lemma}{Lemma}
\theoremstyle{boldspace} 
\theoremstyle{boldspace} \newtheorem{corollary}{Corollary}
\theoremstyle{boldspace} 
\theoremstyle{boldspace} 
\theoremstyle{boldspace}

\DeclareMathOperator{\im}{im}
\DeclareMathOperator{\corank}{corank}

\newcommand{\R}{\mathbb R}

\title{\LARGE \bf
Singular networks and ultrasensitive terminal behaviors
}

\author{Alessio Franci$^{1}$, Bart Besselink$^{2}$, and Arjan van der Schaft$^{2}$
\thanks{$^{1}$Alessio Franci is with the Department of Electrical Engineering and Computer Science, University of Liege, Liege, and with the WEL Research Institute, Wavre, Belgium
        {\tt\small afranci@uliege.be}}%
\thanks{$^{2}$Bart Besselink and Arjan van der Schaft are with the Bernoulli Institute for Mathematics, Computer Science and Artificial Intelligence, University of Groningen, Groningen, the Netherlands. Bart Besselink is also with the Groningen Cognitive Systems and Materials Center (CogniGron), University of Groningen.
        {\tt\small b.besselink@rug.nl}, {\tt\small a.j.van.der.schaft@rug.nl}}%
}

\begin{document}

\maketitle
\thispagestyle{empty}
\pagestyle{empty}

\begin{abstract}
	
Negative conductance elements are key to shape the input-output behavior at the terminals of a network through localized positive feedback amplification. The balance of positive and negative differential conductances creates singularities at which rich, intrinsically nonlinear, and ultrasensitive terminal behaviors emerge. Motivated by neuromorphic engineering applications, in this note we extend a recently introduced nonlinear network graphical modeling framework to include negative conductance elements. We use this extended framework to define the class of singular networks and to characterize their ultra-sensitive input/output behaviors at given terminals. Our results are grounded in the Lyapunov-Schmidt reduction method, which is shown to fully characterize the singularities and bifurcations of the input-output behavior at the network terminals, including when the underlying input-output relation is not explicitly computable through other reduction methods.
\end{abstract}

\section{Introduction}

Characterizing and synthesizing the behavior of resistive networks at given terminals is a fundamental problem of network theory~\cite{bollobas2013modern} with disparate applications~\cite{amin2005realizable,willems2009behavior,wood2013power}. When resistors are linear and their conductance is positive, graph theoretical methods provide a constructive solution to this problem in terms of the network Kron reduction~\cite{dorfler2012kron,van2010characterization}. The generalization of the Kron reduction method to networks of nonlinear resistors with positive differential conductance was recently developed~\cite{van2024kron}.

The goal of this paper is to further generalize, and provide solutions to, the problem of characterizing and synthesizing the terminal behavior of a network by allowing some of the network edges to have negative differential conductance. This generalization is motivated by the broader objective of developing a fundamentally missing physical network theory of excitable behaviors~\cite{sepulchre2018excitable,sepulchre2019control} suitable for neuromorphic engineering applications~\cite{chicca2014neuromorphic,dalgaty2024mosaic,christensen20222022}. 
Whereas the design of neuromorphic excitable signal-processing units is a mature field~\cite{indiveri2011neuromorphic,ribar2019neuromodulation,castanos2017implementing}, the design of excitable networks is still a largely unexplored one. The possibility of rigorously synthesizing large physical networks with tunable excitable behaviors emerging from the balance of positive and negative conductance elements could pave the way to radically new neuromorphic systems designs. To this aim, in this paper we define signed nonlinear networks, i.e., networks with both positive and negative differential conductance edges, and introduce tools from singularity and bifurcation theory to analyze their terminal behaviors.


The paper contributions and its structure are the following. Section~II presents useful preliminaries on signed Laplacians matrices. In this context, we prove a new result on the relation between the corank of a signed Laplacian matrix and that of his Kron reduction. We further introduce the class of singular Laplacian matrices and prove a constructive result to synthesize singular Laplacians with a single negative edge. Section~III leverages the nonlinear network framework of~\cite{van2024kron} to introduce and define signed nonlinear networks. Section~IV defines the class of singular networks and highlights the importance of network parameters in shaping their singular behavior at the given terminals. Section~V shows how Lyapunov-Schmidt reduction and singularity theory methods can be used to characterize the terminal behavior of a singular network, also when this behavior is not explicitly computable. Leveraging these results, we introduce in Section~VI the notions of selectively ultrasensitive terminal behaviors and network bifurcations, and study the bifurcations of networks with a single negative conductance. Finally, Section~VII applies the framework developed in the previous sections to a signed nonlinear network example. Conclusions and perspectives are provided in Section~VIII.

\section{Preliminaries on signed Laplacians}

\subsection{Signed graphs, Laplacians, and their Kron reduction}
\label{sec: signed Laplacians}

A signed graph with $n$ nodes and $m$ edges is defined by an incidence matrix $D\in\R^{n\times m}$ and a {\it signed} weight matrix $W={\rm diag}(w_1,\ldots,w_m)\in\R^{m\times m}$, where $w_j$ can now be both positive or negative. The {\it signed Laplacian} $L$ of a signed graph is defined analogously to the unsigned case,
\begin{equation}\label{eq: signed Laplacian}
	L = D W D^\top\,.
\end{equation}
Observe that also in the signed case $L{\bf 1}_n=0$ and the Laplacian pseudoinverse $L^\dagger$ can be defined as the Moore-Penrose pseudoinverse of $L$~\cite{fontan2023pseudoinverses,chen2020spectral}. 

Consider the partition of the network nodes into $n_B$ boundary (terminal) nodes and $n_C$ central (internal) nodes, with $n_C+n_B=n$. Then, modulo an index reordering, the Laplacian reads
\begin{equation}\label{eq: signe Laplacian partition}
	L= \begin{bmatrix}
		L_{BB} & L_{BC} \\
		L_{CB} & L_{CC} 
	\end{bmatrix},
\end{equation}
where $L_{BB}\in\R^{n_B\times n_B}$ and $L_{CC}\in\R^{n_C\times n_C}$. Suppose that $L_{CC}$ is nonsingular. The $n_B\times n_B$ Schur complement
\begin{equation}\label{eq: signed Laplacian Schur complement}
	\widehat L = L/L_{CC} = L_{BB} - L_{BC}\, L_{CC}^{-1} L_{BC}
\end{equation}
of the signed Laplacian matrix $L$ is again a signed Laplacian matrix~\cite{dorfler2012kron} and we can factorize $\widehat L=\widehat D \widehat{W}\widehat{D}^\top$ to define a reduced signed graph with $n_{B}$ nodes, $\hat m$ edges, incidence matrix $\widehat D$ and weight matrix $\widehat W$. Also in the signed case, the Laplacian $L$ can be interpreted as the linear relation $L z = u$ between the vector of node potentials $z=(z_B,z_C)\in\R^n$ and the vector of nodal currents $u=(u_B,u_C)\in\R^n$ of a signed resistive network in which the conductance of the edge $j$ between node $i$ and node $k$ is $w_j\in\R$. The Schur complement $L/L_{CC}$ is the network {\it signed Kron reduction} associated with the constraint $u_{C}=0$. The following result is new and instrumental to our analysis.

\begin{lemma}\label{lem: L hatL corank}
    Assume that $L_{CC}$ is nonsingular. The signed Laplacian $L$ and its Kron reduction $\widehat L=L/L_{CC}$ satisfy\footnote{Recall that, for a matrix $M$, $\corank M=\dim\ker M$.}
    \begin{equation}\label{eq: L hatL corank}
        \corank L=\corank\widehat L\,.
    \end{equation}
\end{lemma}
\begin{proof}
    Note that $v\in\ker L$ if and only if
    \begin{subequations}\label{eq: corank lem 1}
        \begin{align}
            L_{BB} v_B + L_{BC} v_C &=0\,,\\
            L_{CB} v_B + L_{CC} v_C &=0\,.
        \end{align}
    \end{subequations}
    Hence, if $v\in\ker L$, then $\widehat L v_B=[L_{BB}-L_{BC}L_{CC}^{-1}L_{CB}]v_B = L_{BB}v_B+L_{BC}v_C=0$ and $v_B\in\ker \widehat L$, where the first equality follows from~(\ref{eq: corank lem 1}b) and the second equality follows from~(\ref{eq: corank lem 1}a).

    Conversely, let $v_B\in\ker \widehat L$, i.e., $[L_{BB}-L_{BC}L_{CC}^{-1}L_{CB}]v_B=0$. Then $L_{BB}v_B+L_{BC}v_C=0$, where $v_C=L_{CC}^{-1}L_{CB}v_B$ is uniquely determined by (\ref{eq: corank lem 1}b). Hence, $v=(v_B,v_C)\in\ker L$.

    Thus, for any $v=(v_B,v_C)\in\ker L$, it follows that $v_B\in\ker\widehat L$, and, conversely, for any $v_B\in\ker\widehat L$ there exists unique $v_C$ such that $v=(v_B,v_C)\in\ker L$. Hence $\dim\ker L=\dim\ker\widehat L$.
\end{proof}

\subsection{Effective resistance and singular signed Laplacians}

The effective resistance $r_{ij}$ between two nodes $i$ and $j$ is determined by the voltage $v_i-v_j$ between the two nodes when a unitary current is injected at node $i$ and extracted at node $j$. If $L$ is the, unsigned or signed, Laplacian of the network, then $v_i$ and $v_j$ are obtained by solving $e_{ij}=Lv$ for $v$, where $e_{ij}\in\R^n$ is the vector with $1$ at position $i$, $-1$ at position $j$, and zero elsewhere. If $\corank L=1$, then the solution
\begin{equation*}
    v=L^\dagger e_{ij}
\end{equation*}
is uniquely determined on $\mathbf 1_n^\perp$, hence $v_i-v_j$ is well-defined. This leads directly to the classical characterization~\cite{klein1993resistance} of effective resistance for unsigned graphs
\[
r_{ij} = L^\dagger_{ii} +  L^\dagger_{jj} - 2  L^\dagger_{ij}\,,
\]
where $L$ is an unsigned Laplacian, as well as its signed graphs generalization~\cite{fontan2023pseudoinverses,zelazo2014definiteness,chen2020spectral}
\[
\rho_{ij} = L^\dagger_{ii} +  L^\dagger_{jj} - 2  L^\dagger_{ij}\,,
\]
where now $L$ is a signed Laplacian.

Consider now a signed graph with $n$ nodes, $m$ edges, incidence matrix $D$, weight matrix $W={\rm diag}(w_1,\ldots,w_m)$, Laplacian $L=DWD^\top$, and a single negative edge $l_-$ between node $i$ and node $j$ with weight $w_{l_-}=-k<0$. In particular, $w_l>0$ for all $l\neq l_-$. By removing the edge $l_-$ we obtain an unsigned graph with $n$ nodes, $m-1$ edges, incidence matrix $D_+$, weight matrix $W_+$, and unsigned Laplacian $L_{+} = D_+W_+D_+^\top$. The following result extends and provides a more elementary proof of~\cite[Theorem~III.3]{zelazo2014definiteness}.

\begin{theorem}\label{thm: laplacian singularity}
    Suppose that the graph defined by $L_+$ is connected and let $r_{ij}=(L_+^\dagger)_{ii} +  (L_+^\dagger)_{jj} - 2  (L_+^\dagger)_{ij}>0$.\\[0.1cm]
    \indent 1) If  $k<r_{ij}^{-1}$, then $L$ is positive semidefinite and $\corank L=1$. The restriction of $L$ to $\mathbf 1_n^\perp$ is positive definite.\\[0.1cm]
    \indent 2) If $k=r_{ij}^{-1}$, then $L$ is positive semidefinite, $\corank L=2$, and $L v = 0$, where $v = L_+^\dagger e_{ij}\in \mathbf 1_n^\perp$.\\[0.1cm]
    \indent 3) If $k>r_{ik}^{-1}$, then $L$ is indefinite and $\corank L=1$.
\end{theorem}
\begin{proof}
    Observe that $L=L_+ - k e_{ij}e_{ij}^\top$. Let $v=L_+^\dagger e_{ij}\in \mathbf 1_n^\perp$, $v_{i}-v_{j}=r_{ij}$. Consider $w\in\mathbf 1_n^\perp$ such that $Lw=0$. Then,
    \[
    L_+ w = k e_{ij}e_{ij}^\top w = k (w_i-w_j)e_{ij}.
    \]
    If $w_i-w_j=0$, then $L_+w=0$, which contradicts the hypothesis that the graph defined by $L_+$ is connected~\cite{fiedler1973algebraic}. If $w_i-w_j\neq 0$, then
    \[
    w=k (w_i-w_j)L_+^\dagger e_{ij} = k (w_i-w_j) v\,.
    \]
    Thus, either\footnote{Given $n$-dimensional vectors $v_1,\ldots,v_p$, we use $\R\{v_1,\ldots,v_p\}=\{\alpha_1v_1+\cdots+\alpha_p v_p\:|\:\alpha_1,\ldots,\alpha_p\in\R\}$ to denote their real ``span''.} $\ker L=\R\{\mathbf 1_n\}$ and $\corank L=1$, or $\ker L=\R\{\mathbf 1_n,v\}$ and $\corank L=2$. Observe that
    \[
        Lv = L_+ v - kr_{ij}^{-1} e_{ij} = e_{ij}(1-kr_{ij}^{-1})
    \]
    Thus, if $k=r_{ij}^{-1}$, then $Lv=0$ and $\corank L=2$, while if $k<r_{ij}^{-1}$ or $k>r_{ij}^{-1}$, then $Lv\neq 0$ and $\corank L=1$. Furthermore, by~\cite[Theorem~3.2]{mohar1991laplacian}, decreasing (increasing) the negative conductance weight can only increase (decrease) the eigenvalues of $L$, from which the result follows.
\end{proof}

We call a signed Laplacian as in point 2) of Theorem~\ref{thm: laplacian singularity} a {\it singular signed Laplacian}. Singular signed Laplacians are associated with connected signed graphs, but their second smallest (Fiedler) eigenvalue (equal to the graph algebraic connectivity in the unsigned case) vanishes.

\section{Signed nonlinear networks and their terminal behavior}
\label{sec: signed nonlinear nets}

We build upon and generalize the nonlinear network modeling framework introduced in~\cite{van2024kron} to define nonlinear networks with positive and negative differential conductance edges and start characterizing their input-output behavior at given terminals. The example network in Figure~\ref{fig: net example} is used throughout the rest of paper to illustrate the introduced concepts and results. The $I/V$ characteristic $I=-S(kV,\beta)$ of its negative conductance is defined by the function
\begin{equation}\label{eq: sat cond def}
    S(ky,\beta) = \frac{\tanh(ky-\beta)+\tanh(\beta)}{1-\tanh(\beta)^2}\,,
\end{equation}
where $k$ is the gain of the negative conductance and $\beta\in\R$ is a parameter. Observe that $S(0,\beta)=0$ and $\frac{\partial S}{\partial y}(0,\beta)=k$ for all $\beta\in\R$.

\begin{figure}
    \centering
    \includegraphics[width=0.4\textwidth]{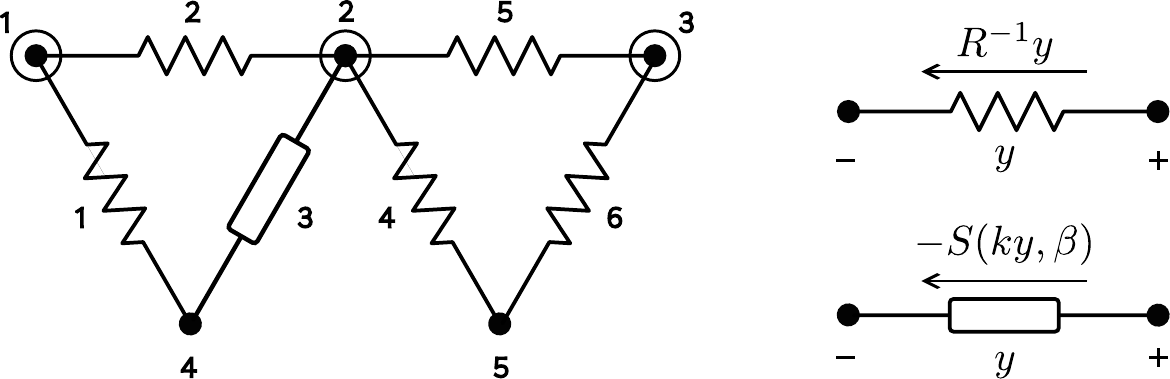}
    \caption{A signed nonlinear network with five nodes, three terminals (encircled vertices), five resistive edges with unitary resistance ($R=1$), and one negative conductance edge (boxed symbol).}
    \label{fig: net example}
\end{figure}

\subsection{Signed nonlinear networks}
\label{sec: nonlinear networks}

A {\it signed nonlinear network} is defined as a connected graph with $n$ vertices, or ``nodes'', $m$ edges, an $n\times m$ incidence matrix $D$, and vectors $z\in\R^n$ and $y=D^\top z\in\R^m$. The components of the vectors $z$ and $y$ are the potentials of the network nodes and the voltages across the network edges, respectively. The relation $y=D^\top z$ expresses the network Kirchhoff's voltage law.

Each edge $j\in\{1,\ldots,m\}$ is endowed with a nonlinear conductance $g_j(y_j)=\frac{\partial G_j}{\partial y_j}(y_j)$ that determines the current flowing through the edge. The scalar conductance potential $G_j(y_j)$ is assumed to be at least twice differentiable. The edge differential conductance is $\frac{\partial g_j}{\partial y_j}$.
A network edge $j$ is positive if its differential conductance is positive, that is, if $g_j(y_j)$ is monotone increasing and $G_j(y_j)$ is convex. A network edge $j$ is negative if its differential conductance is negative, that is, if $g_j(y_j)$ is monotone decreasing and $-G_j(y_j)$ is convex. Let $\mathcal E_+$ be the set of positive edges and $\mathcal E_-$ the set of negative edges. The network potential $G(y)=\sum_{j=1}^m G_j(y_j)$ is the difference of two convex functions
\[
G(y) = \sum_{j\in\mathcal E_+} G_j(y_j) + \sum_{j\in\mathcal E_-} G_j(y_j)\,.
\]

For the network of Figure~\ref{fig: net example}, $\mathcal E_+=\{1,2,4,5,6\}$, $\mathcal E_-=\{3\}$, $G_j(y_j)=\frac{1}{2}y_j^2$ for $j\in\mathcal E_+$, and
\begin{equation}\label{eq: neg cond potential}
    G_3(y_3)=-\frac{\ln(\cosh(ky_3-\beta)+k\tanh(\beta)y_3)}{k(1-\tanh(\beta)^2)}\,.
\end{equation}
Note that $\frac{\partial G_3}{\partial y_3}(y_3) = S(ky_3,\beta)$.

Let $K:\mathbb R^n\to\R$ be defined as
\begin{equation}\label{eq: K unsigned}
	K(z) = G(D^\top z)\,.
\end{equation}
The gradient of $K$ with respect to the node potentials
\begin{equation}\label{eq: unsigned Kirchhoff current}
	\frac{\partial K}{\partial z}(z) = D\frac{\partial G}{\partial y} (D^\top z) = u\,,
\end{equation}
is the vector $u$ of nodal currents at each vertex, which express the network Kirchhoff's current law.

The Hessian of $K$ defines the ($z$-dependent) $n\times n$ {\it signed network Laplacian}
\begin{equation}\label{eq: L unsigned}
	L(z) = \frac{\partial^2 K}{\partial z^2}(z) = D \frac{\partial^2 G}{\partial y^2}(D^\top z) D^\top
\end{equation}
where $\frac{\partial^2 G}{\partial y^2}(D^\top z)$ is the $m\times m$ diagonal matrix of edge differential conductances. Note that $K(z)$, $\frac{\partial K}{\partial z}(z)$, and $L(z)$ are invariant with respect to potential shifts $z\mapsto z+c{\bf 1}_n$.

For the network in Figure~\ref{fig: net example}, it is easy to see that\footnote{Given a differentiable function $f:\R^q\to\R$, we sometimes denote, for $i,j\in\{1,\ldots,q\}$, $f_{x_i}=\frac{\partial f}{\partial x_i}$, $f_{x_ix_j}=\frac{\partial^2 f}{\partial x_i\partial x_j}$, and similarly for higher-order derivatives.}
\begin{equation}\label{eq: ex net laplacian general}
    L(z) = \left[\begin{array}{ccc:cc}
        2 & -1 & 0 & -1 & 0\\
        -1 & 3\!-\!S_y(*) & -1 & S_y(*) -1\\
        0 & -1 & 2 & 0 & -1\\
         \hdashline
        -1 & S_y(*) & 0 & 1\!-\!S_y(*) & 0\\
        0 & -1 & -1 & 0 & 2
    \end{array}\right]
\end{equation}
where $*=(k(z_2-z_4),\beta)$.

\begin{remark}[Network Laplacian and local sensitivity analysis]
    Let $(u,z)$ satisfying $\frac{\partial K}{\partial z}(z)=u$ be a solution of the network behavior. The network Laplacian $L(z)$ characterizes the differential (linearized) sensitivity at $(u,z)$ of nodal currents to infinitesimal changes $\delta z$ in node potentials. The Moore-Penrose inverse $L^\dagger(z)$ of $L(z)$ characterizes the local sensitivity at $(u,z)$ of node potentials to infinitesimal changes $\delta u$ in nodal currents.
\end{remark}

Consider now the partition of the network nodes into $n_B$ boundary nodes, corresponding to the network terminals, and $n_C$ central nodes, corresponding to the network internal nodes, with $n_C+n_B=n$. Let $z_B\in\R^{n_B}$ be the vector of boundary nodes' potentials and $z_C\in\R^{n_C}$ the vector of central nodes' potentials. Modulo a reordering of the network vertices, $z=(z_B,z_C)$. Boundary nodes can be interconnected to external sources or loads. Hence, a nonzero nodal current vector
\begin{equation}\label{eq: terminal behavior}
    \frac{\partial K}{\partial z_B}(z)=u_B
\end{equation}
can flow at boundary nodes. Conversely,
\begin{equation}\label{eq: zero central current}
	\frac{\partial K}{\partial z_C}(z) =0\,,
\end{equation}
i.e., the nodal current at central nodes is zero. At each $z$, the partition into boundary and central nodes induces the block partition~\eqref{eq: signe Laplacian partition} of the network Laplacian, where
\begin{align*}
    L_{BB}(z)&=\frac{\partial^2K}{\partial z_B^2}(z)\,,\quad L_{BC}(z)=\frac{\partial^2K}{\partial z_B\partial z_C}(z)\,,\\
    L_{CB}(z)&=\frac{\partial^2K}{\partial z_C\partial z_B}(z)\,,\quad L_{CC}(z)=\frac{\partial^2K}{\partial z_C^2}(z)\,.
\end{align*}
For the network in Figure~\ref{fig: net example}, this partition is indicated by dashed lines in~\eqref{eq: ex net laplacian general}. Assume that $L_{CC}(z)$ is invertible, i.e.,
\begin{equation}\label{eq: solvable central}
	\det L_{CC}(z)=\det \frac{\partial^2 K}{\partial z_C^2}(z)\neq 0
\end{equation}
Suppose that $\frac{\partial K}{\partial z_C}(z_B^*,z_C^*) =0$ for some $z_B^*,z_C^*$. Then by the implicit function theorem~\eqref{eq: zero central current} can be solved for $z_C$ locally around $(z_B^*,z_C^*)$. Namely, there exists a locally defined twice differentiable function $\bar z_C(z_B)$, with $\bar z_C(z_B^*)=z_C^*$, such that $\frac{\partial K}{\partial z_C}(z_B,\bar z_C(z_B)) \equiv 0$.

\begin{remark}[Nonsingular central Laplacian in the signed and unsigned case]
    If the graph is connected and the network is unsigned, then $\frac{\partial^2 K}{\partial z_C^2}(z)>0$~\cite{van2010characterization} and~\eqref{eq: solvable central} is always satisfied. The same is not true in general in the signed case, as the example network in Figure~\ref{fig: net example} shows. For this network, $\frac{\partial^2 K}{\partial z_C^2}(z)$ becomes singular whenever $S_y(k(z_2-z_4),\beta)=1$ despite the underlying graph being connected.
\end{remark}

The function $\widehat K(z_B) = K(z_B,\bar z_C(z_B))$ defines a reduced model of the network terminal behavior. Indeed,
\begin{align}\label{eq: red terminal behavior}
    \frac{\partial \widehat K}{\partial z_B}(z_B)\!&=\!\frac{\partial K}{\partial z_B}(z_B,\bar z_C(z_B))\! +\! \frac{\partial K}{\partial z_C}(z_B,\bar z_C(z_B))\frac{\partial \bar z_C}{\partial z_B}(z_B)\nonumber\\
    &=\!\frac{\partial K}{\partial z_B}(z_B,\bar z_C(z_B))=u_B\,,
\end{align}
where we used that $\frac{\partial K}{\partial z_C}(z_B,\bar z_C(z_B))=0$ from~\eqref{eq: zero central current}.
Furthermore, for all fixed $z_B$, one can verify~\cite[Equation~(11)]{van2024kron} that $\frac{\partial^2 \widehat K}{\partial z_B^2}(z_B)$ is the Schur complement of the Laplacian $L(z_B,z_C(z_B))$ with respect to $\frac{\partial^2 K}{\partial z_C^2}(z_B,z_C)$, and hence a Laplacian itself~\cite[Theorem~3.1]{van2010characterization}. As detailed in~\cite{van2024kron}, the graph associated to the reduced Laplacian
\begin{equation}\label{eq reduced laplacian}
     \widehat L(z_B)=\frac{\partial^2 \widehat K}{\partial z_B^2}(z_B)
\end{equation}
defines, under certain conditions, a global Kron-reduced network model with only boundary nodes.

\begin{remark}[Non-computability of signed nonlinear Kron-reduced models]
    The function $\bar z_C(z_B)$ is in general hard to compute. In the unsigned case studied in~\cite{van2024kron} it is still possible to globally characterize the network terminal behavior even without an explicit knowledge of $\bar z_C(z_B)$. In the signed case studied here, the same global approach fails in general, which motivates the use of the singularity theory ideas and tools.
\end{remark}

\section{Singular networks and externally singular behaviors}

In this section we introduce and start to develop singularity theory ideas for nonlinear signed networks. 

\subsection{Singular networks}

The network input-output behavior~\eqref{eq: unsigned Kirchhoff current} can be rewritten in the relational form
\begin{equation}\label{eq: implicit network}
   F(z,u,\alpha) = \frac{\partial K}{\partial z}(z) - u = 0 
\end{equation}
where $\alpha$ is a vector of network parameters, like linear conductance values and nonlinear conductance gains. For the network of Figure~\ref{fig: net example}, we have $\alpha=(k,\beta)$. Given the partition $z=(z_B,z_C)$ and $u=(u_B,u_C)$, with $u_C=0$, into boundary and central nodes,~\eqref{eq: implicit network} reads
\begin{align}\label{eq: net kirchhoff current}
    0 & = F(z,u,\alpha) = \begin{pmatrix}
	F_B(z_B,z_C,u_B,\alpha)\\[0.1cm]
	F_C(z_B,z_C,0,\alpha)
    \end{pmatrix}
\end{align}
where
\[
F_i(z_B,z_C,u_i,\alpha) = \frac{\partial K}{\partial z_i}(z_B,z_C,\alpha)-u_i,\quad i=B,C\,.
\]
Let $L(z) = \frac{\partial F}{\partial z} (z,u,\alpha)=\frac{\partial^2 K}{\partial z^2}$ be the (Laplacian) Jacobian of~\eqref{eq: net kirchhoff current}. Assume that $L(z)$ loses rank at a network singular solution $(z^*,u^*)$ for $\alpha=\alpha^*$, namely,
\begin{subequations}
\begin{align}\label{eq: main singularity}
	F(z^*,u^*,\alpha^*) &= 0\\[0.1cm]
	\corank L(z^*) &=2.
\end{align}
\end{subequations}
It follows that there exists $0\neq v\in\mathbf 1_n^\perp$, such that $ L(z^*)\ v = 0$ and $\ker L(z^*)=\mathbb R\{ \mathbf 1_n, v\}$. The solution $(z^*,u^*)$ is called a {\it network singularity} and the network is said to be {\it singular} at $(z^*,u^*)$.

 \begin{remark}[Connections with Catastrophe Theory]
    A network singularity corresponds to a degenerate critical point $z^*$ of $K(z)$. The potential function $K$ is therefore locally flat along $v$ at $z^*$. Flatness of the potential function enables the emergence of network bifurcations and ultra-sensitive behaviors. Small perturbations to any of the network parameters can lead to changes in the convexity of $K$ at $z^*$ and critical points can appear or disappear. The study of the local structure of the critical points of real-valued functions is the main topic of Catastrophe Theory~\cite{thom2018structural}.
\end{remark}

In the following, we will assume, without loss of generality, that $(z^*,u^*,\alpha^*)=0$ and write $L$ instead of $L(0)$ for clarity.

\subsection{Externally singular behaviors}

Assume that the central node behavior is invertible at singularity, that is,
\begin{equation}\label{eq: sing net invC}
	\det \frac{\partial F_C}{\partial z_C}(0,0,0) \neq 0\,.
\end{equation}
As in Section~\ref{sec: nonlinear networks}, the implicit function theorem then implies that, at least locally around the singularity, there exists a differentiable function $\bar z_C(z_B,\alpha)$, defined in a neighborhood of $z_B=0$, such that $\bar z_C(0,0)=0$ and $F_C(z_B,\bar z_C(z_B,\alpha),0,\alpha)\equiv 0$. Replacing this solution in the first component of~\eqref{eq: net kirchhoff current}, we obtain the reduced network terminal behavior
\begin{equation}\label{eq: red net kirchhoff current}
	0= F_B(z_B,\bar z_C(z_B,\alpha),u_B,\alpha) = \widehat{F} (z_B,u_B,\alpha)\,.
\end{equation} 
By Lemma~\ref{lem: L hatL corank}, it follows that the (Laplacian) Jacobian $\widehat L = \frac{\partial \widehat F}{\partial z_B}(0,u_B,\alpha) = L_{BB}-L_{BC}L_{CC}^{-1}L_{CB}$ of~\eqref{eq: red net kirchhoff current} also satisfies $\corank\widehat L=2$ at the network singularity. In particular, there exists $\hat v\in\hat{\mathbf 1}^\perp$ such that $\widehat L\hat v=0$ and $\ker\widehat{L} = \mathbb R\{\hat{\mathbf 1},\hat v\}$, where we denoted $\hat{\mathbf 1}=\mathbf 1_{n_B}$. In other words, the network terminal behavior~\eqref{eq: red terminal behavior} is necessarily singular at a network singularity. A singular network satisfying~\eqref{eq: sing net invC} is said to exhibit an {\it externally singular behavior}.

\section{Lyapunov-Schmidt reduction of singular network terminal behavior}

In this section we compute the Lyapunov-Schmidt reduction~(see~\cite[Section~I.3]{Golubitsky1985} for an introduction to the finite-dimensional version used here) of the input-output behavior of a singular network at the given terminals using either the full network equations~\eqref{eq: net kirchhoff current} or the reduced network equations~\eqref{eq: red net kirchhoff current}. We show in particular that the two reductions coincide (up to a suitable notion of equivalence), which shows that if a Kron reduction of the signed network is known, then Kron and Lyapunov-Schmidt reductions are mutually consistent. In particular, the Lyapunov-Schmidt reduction of the network input-output behavior at its terminals is well-defined, i.e., independent of the used model.

\subsection{Projecting out global shift invariance}

The null eigenvalue/eigenvector pair $L\mathbf 1=0$ is spurious from a singularity theory perspective, in the sense that it reflects the network invariance to uniform shifts of currents and potentials, and not an actual singular behavior. In the following, we therefore restrict~\eqref{eq: net kirchhoff current}, as well as the potential and current vectors $z$ and $u$ to the subspace $\mathbf 1^\perp$ through the projection $\Pi = I - \frac{1}{n}\mathbf 1\mathbf 1^\top$. Because of shift invariance, it is immediate that the projected equations have exactly the same form as~\eqref{eq: net kirchhoff current}. Similarly, $\Pi L\Pi=L$. In the following we throughout restrict all expressions to ${\bf 1}^\perp$ and still use the same notation. In particular, we let
\begin{subequations}\label{eq: net restrict}
    \begin{align}
	&z,u\in\mathbf 1^\perp,\quad F(\cdot,\cdot,\alpha):\mathbf 1^\perp\times \mathbf 1^\perp \to \mathbf 1^\perp,\\  &F_i((\cdot,\cdot),\cdot,\alpha): \mathbf 1^\perp\times \mathbf 1^\perp \to \mathbf 1^\perp,\ i=B,C,\\
	&L  \mathbf 1^\perp \to \mathbf 1^\perp,\quad \corank L = 1, \quad \ker L = \mathbb R\{ v \} \,.
\end{align}
\end{subequations}
Similarly, since the reduced equation~\eqref{eq: red net kirchhoff current} is also shift invariant~\cite{van2024kron}, we let
\begin{subequations}\label{eq: red net restrict}
\begin{align}
	&z_B,u_C\in\mathbf 1^\perp,\quad \widehat F(\cdot,\cdot,\alpha):\hat {\mathbf 1}^\perp\times \hat {\mathbf 1}^\perp \to \hat {\mathbf 1}^\perp,\\
	&\widehat L \hat {\mathbf 1}^\perp \to \hat {\mathbf 1}^\perp,\quad \corank \widehat L = 1, \quad \ker \widehat L = \mathbb R\{ \hat v \} \,.
\end{align}
\end{subequations}

\subsection{Critical eigenvectors of full and Kron-reduced network singularities}


The following lemmas characterize the boundary component $v_B$ of the critical eigenvector $v$ obtained from the full network equations~\eqref{eq: net kirchhoff current} and links it to the critical eigenvector $\hat v$ obtained from the reduced equations~\eqref{eq: red net kirchhoff current}.

\begin{lemma}\label{lem: v1 subspace}
    Let $v=(v_B,v_C)\in\mathbf 1^\perp$. If, $0\neq v\in\ker L$ then $v_B\not\in\mathbb R\{\hat{\mathbf 1}\}$.
\end{lemma}
\begin{proof}
	Assume $v_B =a \hat{\mathbf 1} $.  Then $v$ can be written as$$\mathbf 1^\perp \ni v = a \mathbf 1 + \begin{pmatrix}
	0 \\ \tilde v_C
	\end{pmatrix} =  \begin{pmatrix}
	0 \\ \tilde v_C
	\end{pmatrix},$$ for some $\tilde v_C$, and therefore $a=0$. Since $$L v = L\begin{bmatrix}
	0 \\ \tilde v_C
	\end{bmatrix} =0,$$ using~(\ref{eq: corank lem 1}b) we obtain $\tilde v_C = -L_{CC}^{-1}L_{CB} 0=0$ and therefore $v=0$, which contradicts $v\neq 0$.
\end{proof}
It follows from Lemma~\ref{lem: v1 subspace} that $v_B$ takes the form
\begin{equation}\label{eq: v1 tilde v1}
    v_B = a\hat{\mathbf 1}+\tilde v_B,\quad 0\neq \tilde v_B \in\hat{\mathbf 1}^\perp\,.
\end{equation}
Furthermore, from Lemma~\ref{lem: L hatL corank}, if $0\neq \hat v\in\ker\widehat L$, then there exists $0\neq a_B\in\R$ such that
\begin{equation}\label{eq: v1 and hat v}
    \tilde v_B= a_B\hat v\,.
\end{equation}

\subsection{Lyapunov-Schmidt reduction of the full network equations}

The fundamental idea of the Lyapunov-Schmidt reduction is to solve~\eqref{eq: implicit network}, restricted to $\mathbf 1_n^\perp$ as in~(\ref{eq: net restrict}a), on a complement of $\ker L$ where $L$ is full rank and the implicit function theorem applies. Using the resulting implicit solution, one obtains a scalar equation on $\ker L$ that can be analyzed through singularity theory methods. Let
\[
E^c = \begin{bmatrix}
	v_B v_B^\top & v_B v_C^\top\\[0.1cm]
	v_C v_B^\top & v_C v_C^\top
\end{bmatrix}\quad\mbox{and}\quad E = I - E^c\,.
\]
Observe that the $E^c: \mathbf 1_n^\perp \to \mathbf 1_n^\perp$ is the projection onto $\ker L$, $E: \mathbf 1_n^\perp \to \mathbf 1_n^\perp$ is the projection onto $\im L$, and $\mathbf 1_n^\perp = \im E\oplus\im E^c = \im L\oplus\ker L$. Consider the $n-2$, because of the restriction~(\ref{eq: net restrict}a), independent equations
\begin{equation}\label{eq: full lyap E}
	0 = E F(z,u,\alpha)
\end{equation}
obtained by projecting~\eqref{eq: net kirchhoff current} on $\im L$. Because the linearization $EL=L$ of $E F(z,u,\alpha)$ is invertible on $\im L$, we can split $z=x v+w$, where $x v\in\ker L$, $x\in\R$, and $w=Ez\in\im L$, to solve~\eqref{eq: full lyap E} for $w$ as a function of $x$, $u=(u_B,0)$, and $\alpha$. Let the resulting implicit function be $W(x,u_B,\alpha)$, which satisfies $E F(xv+W(x,u_B,\alpha),u,\alpha)\equiv 0$.
	
On solutions of the network behavior, the implicit function $W(x,u_B,\alpha)$ must be consistent with the network internal behavior $z_C=\bar z_C(z_B)$. In particular, if $F(z,u,\alpha)=0$ then
\begin{align*}
\begin{bmatrix}
	z_B\\z_C
\end{bmatrix} &=
\begin{bmatrix}
	xv_B+W_B(x,u_B,\alpha) \\ xv_C+W_C(x,u_B,\alpha)
\end{bmatrix}  \\
&= \begin{bmatrix}
	xv_B+W_B(x,u_B,\alpha)  \\ \bar z_C(xv_B+W_B(x,u_B,\alpha))
\end{bmatrix},
\end{align*}
which uncovers the following {\it consistency condition} between Lyapunov-Schmidt and Kron reductions
\begin{equation}\label{eq: consistency}
	xv_C+W_C(x,u_B,\alpha) = \bar z_C(xv_B+W_B(x,u_B,\alpha)).
\end{equation}
Inserting~\eqref{eq: consistency} into~\eqref{eq: full lyap E} and using~\eqref{eq: red net kirchhoff current}, we obtain
\[
E\begin{bmatrix}
	\widehat{F}(xv_B\!+\!W_B(x,u_B,\alpha),u_B,\alpha)\\0
\end{bmatrix}\!=\!
\begin{bmatrix}
	\widehat{F}(*)\!-\!v_Bv_B^\top 	\widehat{F}(*)\\
	-v_Cv_C^\top 	\widehat{F}(*)
\end{bmatrix}
\]
where the asterisks in the right-hand side denote the same argument as the left-hand side. Solving the first line for $\widehat{F}(*)$ and replacing in the second leads to $-v_Cv_C^\top 	\widehat{F}(*) =\|v_B\| v_Cv_C^\top \widehat{F}(*) $, which implies $v_Cv_C^\top \widehat{F}(*)=0$. The first line then provides the defining condition
\begin{align}\label{eq: W1 defining}
	0&=(I-v_Bv_B^\top) \widehat{F}(xv_B+W_B(x,u_B,\alpha),u_B,\alpha)\nonumber\\
    &=(I- a_B^2\hat v\hat v^\top) \widehat{F}(a_Bx\hat v+W_B(x,u_B,\alpha),u_B,\alpha)
\end{align}
of the boundary component $W_B(x,u,\alpha)$ of $W$. Note that for the second equality we used~\eqref{eq: v1 and hat v} and~(\ref{eq: red net restrict}a). We can plug this solution into the complementary rank~1 system of equations $E^cF(z,u,\alpha)=0$ on $\ker L$ to obtain
\[
	E^cF(z,u,\alpha) = \begin{pmatrix}
		v_Bv_B^T \widehat{F}(xv_B+W_B(x,u_B,\alpha),u_B,\alpha)\\
		v_Cv_C^T \widehat{F}(xv_B+W_B(x,u_B,\alpha),u_B,\alpha)
	\end{pmatrix}.
\]
Finally, the Lyapunov-Schmidt reduction of~\eqref{eq: net kirchhoff current} is
\begin{align}\label{eq: full lyap schmidt}
	f(x,u_B,\alpha) &= v^\top E^cF(z,u,\alpha) \nonumber \\
	&=v_B^\top \widehat{F}\bigl(xv_B+W_B(x,u_B,\alpha),u_B,\alpha\bigr)\nonumber\\
    &=a_B\hat v^\top \widehat{F}\bigl(a_Bx\hat v+W_B(x,u_B,\alpha),u_B,\alpha\bigr)\,,
\end{align}
where $W_B(x,u_B,\alpha)$ satisfies~\eqref{eq: W1 defining} and we used again~\eqref{eq: v1 and hat v} and~(\ref{eq: red net restrict}a). 

\subsection{Lyapunov-Schmidt reduction of the reduced network equations}

Let $\widehat E^c=\hat v\hat v^\top:\hat{\mathbf 1}^\perp\to\hat{\mathbf 1}^\perp$ and $\widehat E=I-\widehat E^c\hat{\mathbf 1}^\perp\to\hat{\mathbf 1}^\perp$ be the projection onto $\ker \widehat L$ and $\im \widehat L$, respectively. Consider the splitting $z_B=x\hat v+\hat w$, where $x\hat v\in\ker\widehat{L}$, $x\in\R$, and $\hat w\in\im\widehat{L}$. Invoking the implicit function theorem, the $n_B-2$ independent equations $\widehat E\widehat{F}(z_B,u_B,\alpha)=0$ can be solved for $\hat w$ in $\im \widehat{L}$, which leads to the implicit function $\widehat W(x,u_B,\alpha)$ satisfying
\begin{equation}\label{eq: hat W defining}
	(I-\hat v\hat v^\top) \widehat{F}\bigl(x\hat v+\widehat W(x,u_B,\alpha),u_B,\alpha\bigr)=0
\end{equation}
The remaining rank 1 system of equations $\widehat{E}^c\widehat{F}(z_B,u_B,\alpha)=0$ on $\ker \widehat L$, leads to the Lyapunov-Schmidt reduction of~\eqref{eq: red net kirchhoff current}, which reads
\begin{align}\label{eq: red lyap schmidt}
\hat f(x,u_B,\alpha) &= \hat v^\top \widehat E^c\widehat F(z_B,u_B,\alpha) \nonumber \\
&=\hat v^\top \widehat{F}\bigl(x\hat v+\widehat W(x,u_B,\alpha),u_B,\alpha\bigr)\,,
\end{align}
and where $\widehat W(x,u_B,\alpha)$ satisfies~\eqref{eq: hat W defining}.

\subsection{Consistency of the Lyapunov-Schmidt reduction of full and reduced network behaviors}

\begin{theorem}\label{thm: main equivalence}
Let $a_B$ be defined in~\eqref{eq : v1 and hat v}. Then
    \begin{align*}
        \hat f(x,u_1,\alpha) &= a_B^{-1}f(x,u_B,\alpha)\,.
    \end{align*}
\end{theorem}
\begin{proof}
    Using the change of variables $ \tilde x = a_Bx$, \eqref{eq: W1 defining} and~\eqref{eq: full lyap schmidt} can be rewritten as
\begin{align*}
    0&=(I-a_B^2\hat v\hat v^\top)\widehat F(\tilde x\hat v+W_B(\tilde x/a_B,u_B,\alpha),u_B,\alpha)
\end{align*}
and
\begin{align*}
    g(\tilde x,u_B,\alpha)&=a_B\hat v^\top\widehat F(\tilde x\hat v+W_B(\tilde x/a_B,u_B,\alpha),u_B,\alpha)\,.
\end{align*}
Comparing with~\eqref{eq: hat W defining} and~\eqref{eq: red lyap schmidt}, it follows immediately that
\begin{align*}\label{eq: main equivalence}
    \widehat W(x,u_B,\alpha)&=W_1(a_B^{-1}x,u_B,\alpha)\\
    \hat f(x,u_B,\alpha) &= a_B^{-1}f(x,u_B,\alpha) \qedhere
\end{align*}
\end{proof}

Theorem~\ref{thm: main equivalence} implies  that $f(x,u_B,\alpha)$ and $\hat f(x,u_B,\alpha)$ define {\it strongly equivalent} bifurcation problems~\cite{Golubitsky1985}. In other words, the nonlinear terminal behavior of a network close to a singularity can equivalently be studied through the Lyapunov-Schmidt reduction of its full~\eqref{eq: net kirchhoff current} or reduced~\eqref{eq: red net kirchhoff current} models.

\begin{remark}[Application of Theorem~\ref{thm: main equivalence}]
The function $\bar z_C(z_B)$ and therefore the reduced network equations $\widehat F(z_B,u_B,\alpha)=0$ are hard to compute in general. Theorem~\ref{thm: main equivalence} ensures that the reduced boundary behavior, its singularities, and bifurcations can be fully characterized by the standard application of the Lyapunov-Schmidt reduction to the full network equations.
\end{remark}

\section{Ultrasensitive behaviors and network bifurcations}

\subsection{Ultrasensitive behaviors}

Unpacking the definition of the function $\widehat F$ in~\eqref{eq: full lyap schmidt}, one obtains
\begin{align*}
    f(x,u_B,\alpha) &= \hat v^\top\frac{\partial K_B}{\partial z_B}(z_B,\bar z_C(z_B)) + \hat v^\top u_B\\
    &=\tilde f(x,\alpha) + \hat v^\top u_B\,,
\end{align*}
where $\tilde f(x,\alpha)=\hat v^\top\frac{\partial K_B}{\partial z_B}(z_B,\bar z_C(z_B))$. As shown in~\cite[Section I.3]{Golubitsky1985}, the network singularity condition implies that
\[
\frac{\partial f}{\partial x}(0,0,0)=\frac{\partial \tilde f}{\partial x}(0,0) = 0.
\]
The linearization of the equation $0=\tilde f(x,\alpha) + \hat v^\top u_B$ at $(x,u_B,\alpha)=(0,0,0)$ then reads
\[
0=0\cdot x + \hat v^\top u_B\,,
\]
which reveals that, close to a singularity, the linearized input-to-state sensitivity blows up. The behavior is {\it ultra-sensitive}. Furthermore, this ultra-sensitive behavior is {\it selective}, in the sense that all input vectors $u_B$ that are orthogonal to the critical eigenvector $\hat v$ are filtered out by the network behavior.

\subsection{Network bifurcations}

Consider the parameterization
\begin{equation}\label{eq: bif parameterization}
    (u,\alpha)=(u(\lambda),\alpha(\lambda)),\quad \lambda\in\R
\end{equation}
of the network input and parameters, with $(u(0),\alpha(0))=(0,0)$. As the {\it bifurcation parameter} $\lambda$ crosses zero, the network crosses its singularity and undergoes a {\it bifurcation}. With a little abuse of notation, let
\[
f(x,\lambda) = f(x,u_B(\lambda),\alpha(\lambda))
\]
be the network Lyapunov-Schmidt reduction~\eqref{eq: full lyap schmidt} at $(z,u,\alpha)=(0,0,0)$. The set
\begin{equation}\label{eq: bifurcation diagram}
    \mathcal B=\{(x,\lambda)\,:\,f(x,\lambda)=0\}
\end{equation}
is called the {\it bifurcation diagram} of $f$. Since the solutions to $f(x,u,\alpha)=0$ are in one-to-one correspondence with the solutions of~\eqref{eq: net kirchhoff current} through the smooth mapping $z=xv+W(x,u_B,\alpha)$, $\mathcal B$ provides a qualitative description of how network solutions change close to a network singularity as inputs and parameters change through $\lambda$. In particular, the {\it recognition problem} formulation~\cite[Chapter~II]{Golubitsky1985} allows one to characterize $\mathcal B$ via the Taylor expansion of $f$ obtained through implicit differentiation methods~\cite[Section~I.3.e]{Golubitsky1985}.

\subsection{Network bifurcations with a single negative conductance}

Consider a signed nonlinear network with $n$ nodes, partitioned as above into $n_B$ boundary nodes and $n_C$ central nodes, $m$ edges, and with a single negative conductance
\begin{equation}
  I_{l_-}=-k\frac{\partial G_{l_-}}{\partial y_{l_-}}(y_{l_-}),\quad \frac{\partial G_{l_-}}{\partial y_{l_-}}(0)=0  
\end{equation}
on edge $l_-$ between nodes $i$ and $j$, where $G_{l_-}(y_{l_-})$ is a convex function and $k$ is the tunable gain of the negative conductance. Without loss of generality, let $\frac{\partial^2 G_{l_-}}{\partial y_{l_-}^2}(0)=1$. Assume, for simplicity and conciseness of exposition, that all the positive conductances are linear and fixed, i.e., $G_l(y_l)=\frac{1}{2}R_l^{-1}y_l^2$ for all $l\in\{1,\ldots,m\}\setminus l_-$. Hence, the network obtained by removing the negative conductance edge is a linear and fixed resistive network. 

The resulting signed nonlinear network $\mathcal N$ is parameterized by a single parameter $\alpha=k$. Let $L(z)$ be the network Laplacian and $L_+$ be the Laplacian of the linear resistive network obtained by removing the negative conductance edge. Also, let $g_-(y_{l_-})=\frac{\partial G_{l_-}}{\partial y_{l_-}}(y_{l_-})$. The proof of the following theorem is omitted for space constraints but follows directly from Theorem~\ref{thm: laplacian singularity} and the elementary algebraic manipulations required to apply the implicit differentiation formulae in~\cite[Section~I.3.e]{Golubitsky1985}.

\begin{theorem}\label{thm: single neg res}
    Assume that the graph associated with $L_+$ is connected and let $r_{ij}=v_i-v_j$ with $L_+v=e_{ij}$. Then, $\mathcal N$ is singular at $(z,u,k)=(0,0,r_{ij}^{-1})$ and $\ker L(0)=\R\{\mathbf 1_n,v\}$. Furthermore, if $\frac{\partial^2 K }{\partial z_C^2}$ is invertible, then the Lyapunov-Schmidt reduction~\eqref{eq: full lyap schmidt}, or equivalently~\eqref{eq: red lyap schmidt}, satisfies, at $(x,u_B,k)=(0,0,r_{ij}^{-1})$,
   \begin{align*}
       &f=f_x=f_k=0,\quad f_{xx}=-g_-''(0)r_{ij}^2,\\
       &f_{kx}=-r_{ij}^2,\quad f_{xxx}=-g_-'''(0)r_{ij}^3\,.
   \end{align*}
\end{theorem}

\begin{corollary}\label{cor: single neg res}
    Let $k=\lambda$ be the bifurcation parameter of the nonlinear behavior of the signed nonlinear network $\mathcal N$. If $g_-''(0)\neq0$,
    then the input-output terminal behavior of $\mathcal N$ undergoes at $(x,u_B,k)=(0,0,r_{ij}^{-1})$ a {\it transcritical} bifurcation. If $g_-''(0)=0$ and $g_-'''(0)\neq0$, then the bifurcation at $(x,u_B,k)=(0,0,r_{ij}^{-1})$ is a {\it pitchfork}. Furthermore, the pitchfork is supercritical if $g_-'''(0)<0$ and subcritical if $g_-'''(0)>0$. Finally, if $g_-''(0)=g_-'''(0)=0$, then $\mathcal N$ undergoes at $(x,u_B,k)=(0,0,r_{ij}^{-1})$ a bifurcation of codimension$>2$.
\end{corollary} 
\begin{proof}
    The case $g_-''(0)\neq0$ follows from Theorem~\ref{thm: single neg res} and~\cite[Proposition~II.9.3]{Golubitsky1985}. The two cases $g_-''(0)=0$, $g_-'''(0)\neq0$, follows from Theorem~\ref{thm: single neg res} and~\cite[Proposition~II.9.2]{Golubitsky1985}. The case $g_-''(0)=g_-'''(0)=0$ follows from the Classification Theorem~\cite[Theorem~IV.2.1]{Golubitsky1985}.
\end{proof}

Theorem~\ref{thm: single neg res} and Corollary~\ref{cor: single neg res} reveal how the nonlinear properties of the negative conductance $g_-(y_{l_-})$, and in particular its derivatives computed at $y_{l_-}=0$, determine bifurcation behavior of the network at its singularity.

\section{Numerical example}

\begin{figure}
\centering
    \includegraphics[width=0.33\textwidth]{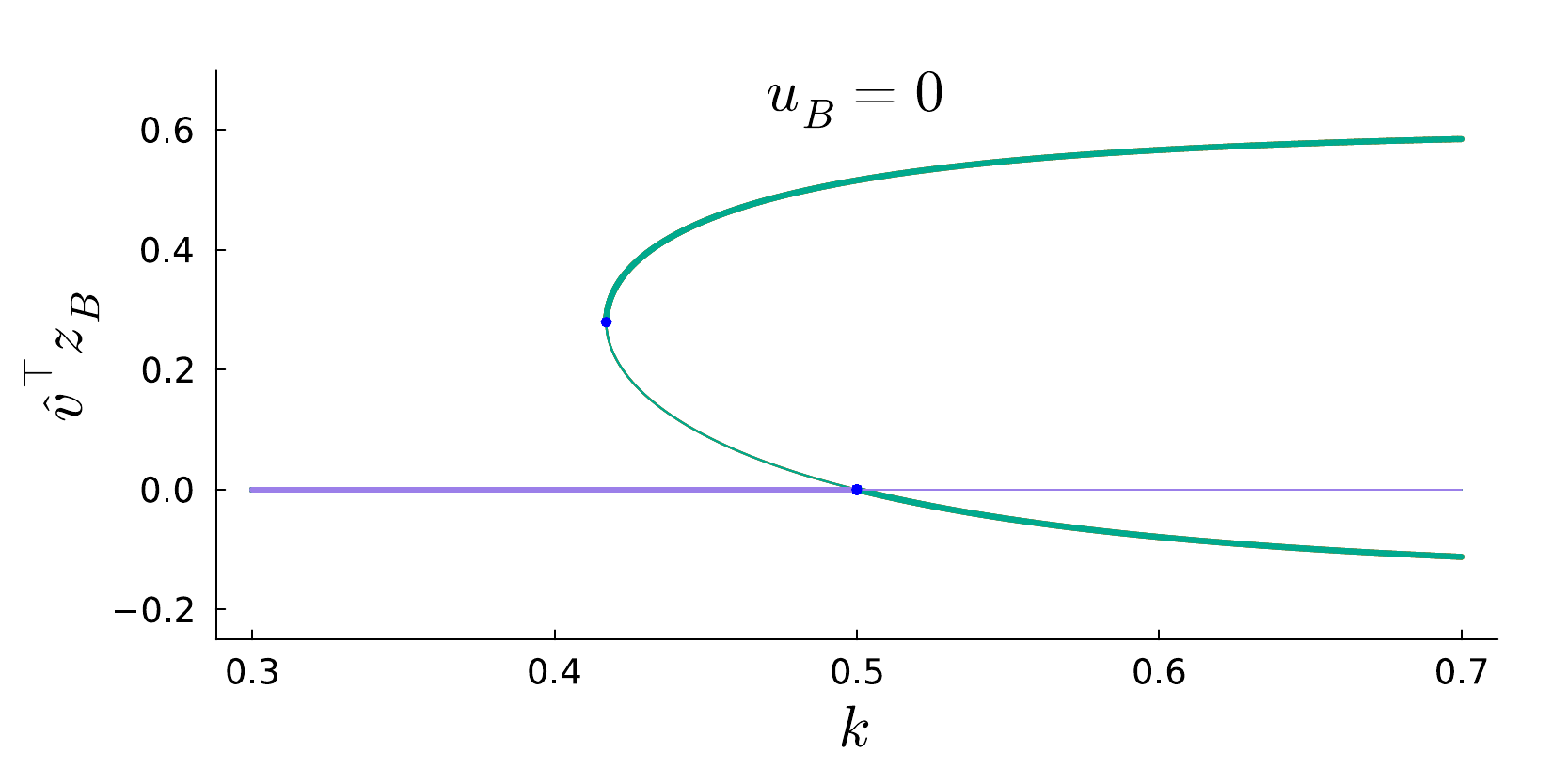}
    \includegraphics[width=0.33\textwidth]{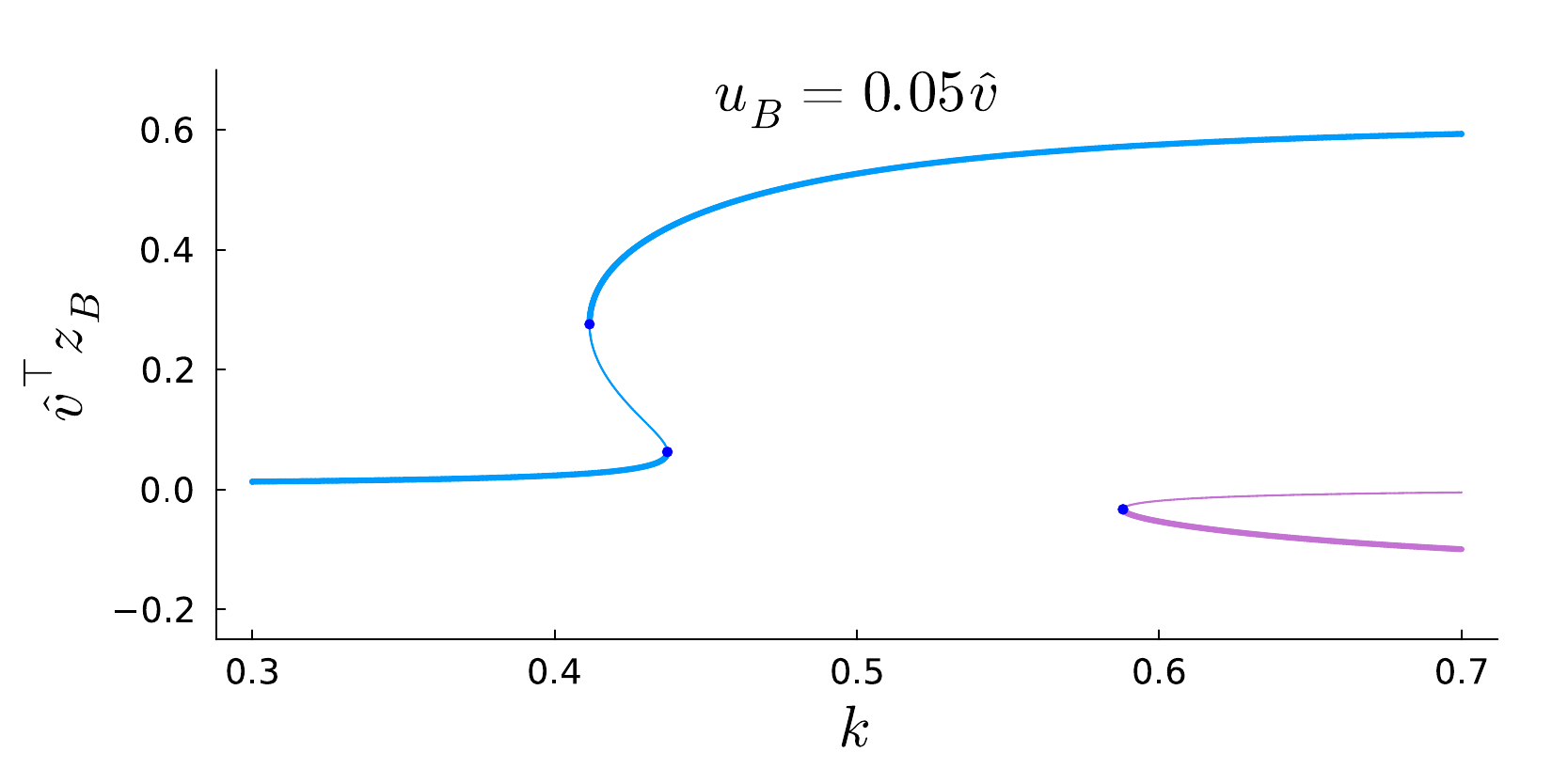}
    \includegraphics[width=0.33\textwidth]{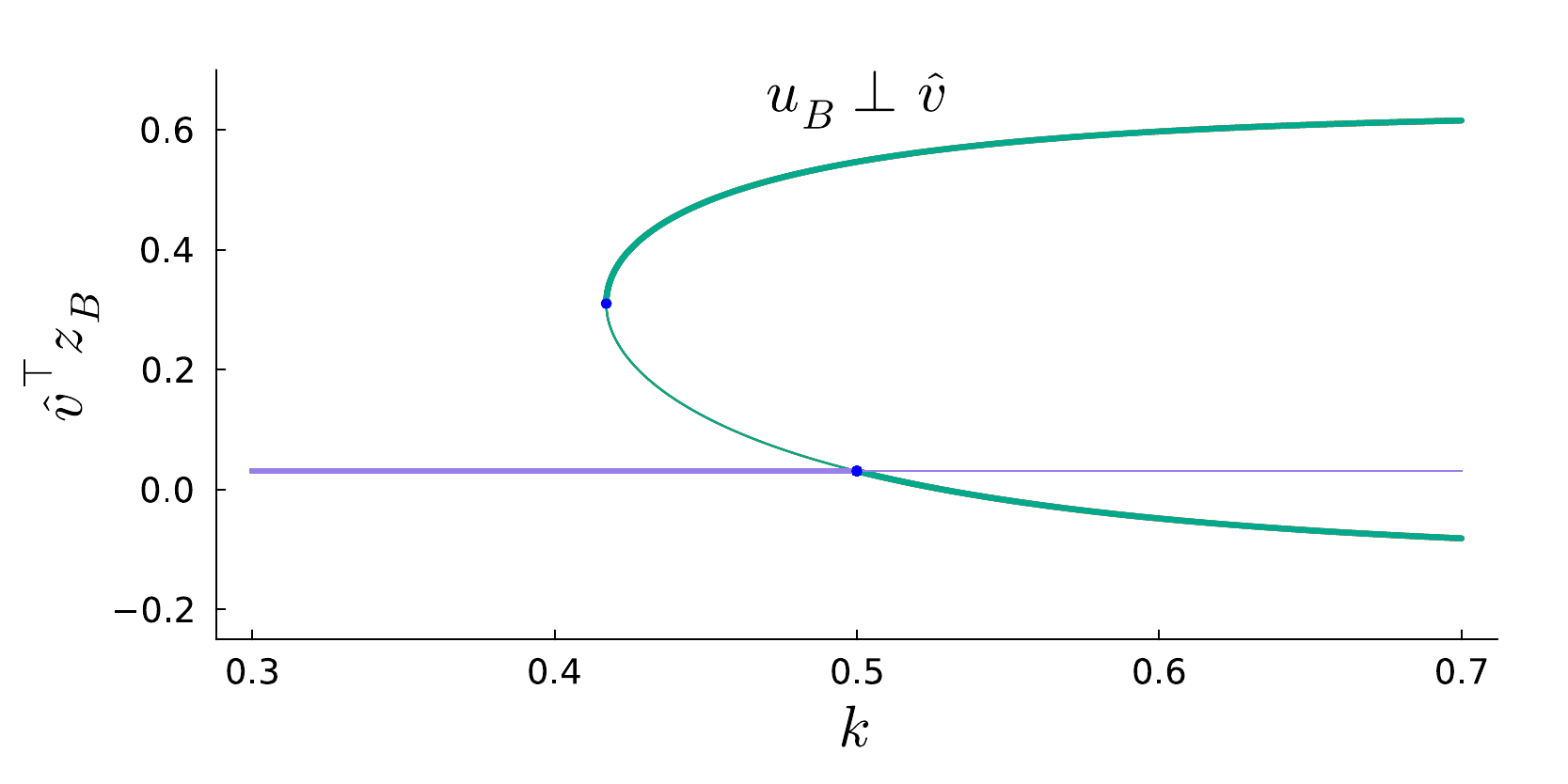}
\caption{Bifurcation diagrams (projected on $\hat v$) with respect to the negative conductance gain $k$ of the terminal behavior of the network in Figure~\ref{fig: net example} for different terminal current vectors $u_B$ and $\beta=0.5$. Branches of stable (unstable) equilibria are indicated by thick (thin) lines.} 
\label{fig: net bif}
\end{figure}

Consider the network in Figure~\ref{fig: net example}. A direct application of Theorem~\ref{thm: single neg res} shows that, for all $\beta\in\R$, the network is singular at $(z,u_B,k)=(0,0,1/2)$. Furthermore, if $\beta\neq0$, then $g_-''(0)=\frac{\partial^3 G_3}{\partial y_3^3}(0)\neq 0$, where $G_3(y_3)$ is defined in~\eqref{eq: neg cond potential}, hence we expect the bifurcation at $(z,u_B,k)=(0,0,1/2)$ to be transcritical. A final prediction is that, close to singularity, the network is insensitive to inputs $u_B\in \hat v^\perp$ and ultra-sensitive to inputs $u_B$ such that $\hat v^\top u_B\neq 0$. 

Figure~\ref{fig: net bif} confirms and illustrates these predictions. The bifurcation diagrams were generated in Julia BifurcationKit~\cite{veltz:hal-02902346} by assuming a small parasitic capacitance at each node. The top panel of Figure~\ref{fig: net bif} shows that when no boundary currents are applied the network exhibits a transcritical bifurcation at $k=0.5$, as predicted. At this bifurcation the ground network state $z=(z_B,\bar z_C(z_B))=(0,0)$ becomes unstable and, even in the absence of any boundary current, new stable configurations of network node potentials emerge along the critical eigenvector $\hat v$. The center panel of Figure~\ref{fig: net bif} shows how even small boundary current vectors aligned with the critical eigenvector $\hat v$ can induce major changes in the network terminal behavior. This is reflected in the markedly different bifurcation diagram as compared to the zero boundary current case. The big change in the network potential bifurcation diagram in response to small boundary current vectors illustrates the ultra-sensitivity of the network terminal behavior close to its singularities. Conversely, the bottom panel of Figure~\ref{fig: net bif} shows how even large boundary current vectors that are orthogonal to $\hat v$ have no effects on the network bifurcation diagram, which illustrates the selectivity of the network ultra-sensitive behavior.

\section{Conclusions and perspectives}

We introduced a class of models and a set of tools to pave the way towards the modeling, analysis, and synthesis of physical networks with intrinsically nonlinear, and in particular excitable, behavior at their terminals. The Lyapunov-Schmidt reduction, and its consistency with the reduced network behavior at the terminals, are the new key elements supporting the resulting nonlinear network theory.

The results and ideas we presented set the foundations for a number of generalizations and extensions.
We mention some of them here.
In the same sprit as~\cite{huijzer2023synchronization}, the generalization to networks of static conductances and grounded capacitors is natural and will be presented in an extended version of this work. A more involved extension, still under development, is to consider networks with dynamic edges. In this direction, the (Henkel) operator theoretical approach introduced in~\cite{chaffey2023relaxation} seems a promising one. Finally, how these ideas apply to networks of heterogeneous components, like multi-physics systems, is an important and still largely unexplored problem, for which the port-Hamiltonian framework~\cite{van2000l2,van2024reciprocity} constitutes probably the most natural modeling substrate.

\addtolength{\textheight}{-12cm}   





\bibliography{biblio}
\bibliographystyle{IEEEtran}

\end{document}